\documentclass[12pt, reqno, a4paper]{amsart}
\usepackage{amsmath,mathrsfs}
\usepackage{amssymb}
\usepackage{calligra}
\usepackage{dsfont}

\newcommand\NoBlackBoxes{\global\overfullrule0pt}
\NoBlackBoxes

\theoremstyle{plain}\newtheorem{theo}{Theorem}
\theoremstyle{plain}\newtheorem{cor}{Corollary}
\theoremstyle{definition}\newtheorem{rem}{Remark}
\theoremstyle{plain}\newtheorem{defi}{Definition}[section]
\theoremstyle{plain}
\theoremstyle{plain}\newtheorem{prop}[defi]{Proposition}
\theoremstyle{definition}

\newcommand{\N}{{\mathbb{N}}}

\begin{document}
\title[Hitting times for random walks on random hypergraphs]{Hitting times, commute times, and cover times for random walks on random hypergraphs}

\author[Amine Helali]{Amine Helali}\thanks{Amine Helali would like to thank the University of M\"unster for the hospitality during his stay in June, July 2018, and January, February 2019.
His work was supported by the "Direction Europe et Internationale" of the Universit\'e de Bretagne Occidentale, Brest, France.}

%
\address[Amine Helali]{Laboratiore de Math\'{e}matiques de Bretagne Atlantique UMR 6205, UFR Sciences et Techniques, Universit\'{e} de Bretagne Occidentale, 6 Avenue Le Gorgeu, CS 93837, 29238 Brest, cedex 3, France.
\newline
Laboratoire MODAL'X, UFR SEGMI, Universit\'{e} Paris Nanterre, 200 Avenue de la R\'epublique 92001 Nanterre, France.}


\email[Amine Helali]{ amine.helali001@gmail.com}

\author[Matthias L\"owe]{Matthias L\"owe }\thanks{Research of the second author was
funded by the Deutsche Forschungsgemeinschaft (DFG, German Research Foundation) under Germany 's Excellence Strategy
EXC 2044 390685587, Mathematics M\"unster: Dynamics Geometry -Structure}
\address[Matthias L\"owe]{Fachbereich Mathematik und Informatik,
Universit\"at M\"unster,
Einsteinstra\ss e 62,
48149 M\"unster,
Germany}

\email[Matthias L\"owe]{maloewe@uni-muenster.de}

\subjclass[2010]{60B20, 05C81, 05C80}
\keywords{random walks on random hypergraphs, hitting time, cover time, commute time,  spectrum of random graphs}


\begin{abstract}
We consider simple random walk on the structure given by a random hypergraph in the regime where there is a unique giant
component. Using their spectral decomposition we give the asymptotics for
hitting times, cover times, and commute times and show that the results obtained for random walk on random graphs are universal.
\end{abstract}
%

\maketitle
\section{Introduction}
Random walks on random graphs have been an active research area in probability theory for a long time, see e.g.
\cite{doylesnell, lovaszgraphs, woess}. Besides being a field that poses interesting question in its own right,
they have also been a key tool to understand the properties of random graphs, especially close to the point of
phase transition (for a very readable survey see the recent monograph  \cite{vdH17}).
These so-called exploration processes have been transferred to the investigation of random hypergraphs,
see e.g. \cite{bollobasriordan2,bollobasriordan3}. This fact may motivate the study of random walks on random
hypergraph structures as well.

However, already \cite{CFR2},\cite{CFR1} studied the so-called cover time of random walk on a random uniform hypergraph.
They considered the following model: Take $H$ uniformly at random from all $r$-regular, $d$-uniform hypergraphs.
Hence every vertex $v\in V:=\{1, \ldots, n\}$ is contained in $r$ hyperedges and for all hyperedges $e \in E$
it holds $|e|=d$. Colin, Frieze and Radzik analyze simple random walk on the resulting structure, i.e. if the random walk is in a vertex $v$ at time $t \in \N$, for the vertex at time $t+1$ it selects a hyperedge $e$, such that $v \in e$ and then it selects any $w \neq v$ in $e$ with probability $\frac{1}{d-1}$ and walks there. For this walk the authors analyze the so-called  cover time, i.e. the expected time it takes the walk to see every vertex of $V$. They show that this time $C(H)$ is of order
$ \left(1  + \frac{1}{(r-1)(d-1)-1} \right) n \log n (1+o(1)).
$

Inspired by the results in \cite{LT14} we will study the hitting times, commute times, and cover times for random walks on random hypergraphs.
We will refrain from considering regular hypergraphs, but  stick with uniform hypergraphs setting.
This means, 
the underlying structure will consist of a realization of a random $d$-uniform hypergraph $H$ on $V=\{1, \cdots, n\}$, i.e. all edges $e \in \binom{V}{ d}$ are selected independently and with equal probabilities $p$. This model is known $H(n, p)$ and $E$  is the edge set of the hypergraph. We assume that $p=p_n \gg  \frac{\log^4 n}{n^{d-1}}$ (where we write $a_n \gg b_n$, if and only if $\frac{b_n}{a_n} \to 0$), such that, with probability converging to $1$, $H$ is connected. 
All the probabilities considered below are to be understood conditionally on the event that $H$ is connected.

On this structure we will consider simple random walk as described above.
This random walk, that we will henceforth call $(X_i)$, can either be considered as a random walk on the multi-graph $G=(V,\widetilde{E}$) associated with $H$, i.e. if $v, w\in V$ are in $k$ hyperedges, then there are $k$ edges connecting $v$ and $w$ in $\tilde E$. Alternatively, we can consider the random walk on the weighted graph, where the weight of an edge $\{v,w\}$ is the number of hyperedges containing both $v$ and $w$. The invariant measure of the walk is
$$\pi(i) = \frac{  \sum_{e\in E}  \mathds{1}_{\{i \in e\}}}{ d |E|  }  = \frac{d(i)}{ \sum_{j\in V} d(j)}$$
where the degrees $d(i)$ are counted in the multi-graph interpretation.

\section{Hitting times}
For the random walk $(X_i)$ consider the following  quantities.
Let 
$H_{ij}$ be the expected time it takes the walk to reach vertex $j$ when starting from vertex $i$.
Moreover, let
$$
H_j :=  \sum_{i\in V} \pi(i) H_{ij} \quad \mbox{ and}  \quad H^i :=  \sum_{j \in V} \pi(j) H_{ij}
$$
be the {\it average target hitting time} and {\it the average starting hitting time}, respectively (these names are taken from \cite{levinperes}).
Note that both, $H_j$ and $H^i$ are expectation values in the random walk measure, but random variables with respect to the realization of the random hypergraph. Also note that, in general, $H_j$ and $H_i$ will be different.

In \cite{LT14} the same quantities were studied for random graphs instead of random hypergraphs and it was shown that $H_j =n(1+o(1))$ asymptotically almost surely (a.a.s., for short), which means that the probability that a vertex $j$ admits $H_j$ that is not of this order, vanishes for $n \to \infty$. This result confirmed a prediction in the physics literature (see \cite{Sood}). The aim of the present note is to generalize this result to our random hypergraph setting. Our results can hence be understood as a universality statement about  random graphs and hypergraphs. They also may be interpreted as a generalization of the results in \cite{LT14} to weighted graphs and multi-graphs. A key difference between the random graph case and our situation, however, is not only that we may have multiple edges connecting two nodes, but also that these edges are no longer independent. Moreover, a key tool in \cite{LT14} is the analysis of the spectrum of a random graph taken from \cite
 {Erdoesetal11}. This is not available in our setting.

We will thus to give asymptotic results for $H_j$ and $H^i$. To this end, we will derive a different representation of $H_j$ and $H^i$ as in \cite{lovaszgraphs}.
Let $B:=\sqrt{D} A \sqrt{D}$ be the graph Laplacian of the hypergraph structure we realize. Here $D:=\left(\mbox{diag}(\frac{1}{d_i})\right)_{i=1}^n$ and $A=(a_{ij})$ is the adjacency matrix of the multi-graph $\widetilde{G}=(V, \widetilde{E})$.
Thus,
$$a_{ij}= \sharp \{  e \in \widetilde{E}: \,\, e=\{i, j\} \}$$
and $$B= \left(  \frac{a_{ij}}{\sqrt{d_i}\sqrt{d_j}} \right).$$
Let $$\lambda_1 \geq \lambda_2 \geq \cdots \geq \lambda_n $$ be the eigenvalues of $B$.
$w:=(\sqrt{d_1}, \cdots, \sqrt{d_n})$ satisfies
$$Bw=\sum_{i=1}^n a_{ij} w_j = \sum_{i=1}^n a_{ij} d_i = \frac{d_i}{\sqrt{d_i}} = \sqrt{d_i}.
$$
Thus, $\lambda_1 = 1$ is an eigenvalue for the matrix $B$ and by the Perron-Frobenius theorem it is the largest one. We will always normalize the eigenvectors $v_k$ to the eigenvalues $\lambda_k$ to length one such that, in particular,
$$
v_1 := \frac{w_*}{\sqrt{2|\widetilde{E}|}}= \left(\sqrt{\frac{d_j}{2|\widetilde{E}|}} \right)_{j=1}^n.
$$
In general, the matrix of the eigenvector is orthogonal and the scalar product of two eigenvectors $v_i$ and $v_j$  satisfies $\langle v_i, v_j\rangle =\delta_{ij}$. In particular, for $v_1$ we obtain:
$$
0=<v_k, v_1>= \frac{1}{2 |\widetilde{E}|}  \sum_{j=1}^n v_{k, j} \sqrt{d_j} \mbox{ for }  k\neq 1 \quad \mbox{and  }
\sum_{j=1}^n v_{k, j}^2 =  \sum_{k=1}^n v_{k, j}^2=1.
$$
A key observation for our context is that hitting times possess a spectral decomposition as was given by Lov\'asz (see \cite{lovaszgraphs}) in the following theorem.
\begin{theo}\cite[Theorem 3.1]{lovaszgraphs}\label{Lovasztheo}
The expected hitting times have the following spectral decomposition
\begin{equation}\label{specrep}
H_{ij}=2 |\widetilde{E}| \sum_{k=2}^n \frac{1}{1- \lambda_k} \left(   \frac{v_{k, j}^2}{d_j}   -   \frac{v_{k, i}  v_{k, j}}{\sqrt{d_i d_j}}     \right).
\end{equation}
\end{theo}
As a matter of fact, Lov\'asz proves this theorem just for ordinary graphs. It is, however, simple matter to check that it easily translates to multi-graphs.
Theorem \ref{Lovasztheo} allows to also give a spectral representation of the average target hitting time and the average starting hitting time $H_j$ and $H^i$. Indeed, using Theorem \ref{Lovasztheo} together with the orthognality of the eigenvectors gives
\begin{align}
H_j &= \sum_{i=1}^n \pi(i) H_{ij} =\sum_{i=1}^n \sum_{k=2}^n \frac{1}{1-\lambda_k} \left(  v_{k, j}^2 \frac{d_i}{d_j} - v_{k,i} v_{k, j} \sqrt{\frac{d_i}{d_j}} \right) \nonumber \\
&=  \left( \frac{1}{d_j} \sum_{i=1}^n d_i  \right)  \left( \sum_{k=2}^n \frac{1}{1-\lambda_k} v_{k,j}^2  \right) - \sum_{k=2}^n \frac{1}{\sqrt{d_j}} \frac{v_{k,j}}{1-\lambda_k} \sum_{i=1}^n v_{k,i} \sqrt{d_i} \nonumber \\
&= \frac{2 |\widetilde{E}|}{d_j}  \sum_{k=2}^n \frac{1}{1-\lambda_k} v_{k,j}^2  -
\sum_{k=2}^n  \frac{ \sqrt{2 |\widetilde{E}|}}{ \sqrt{d_j}}  \frac{v_{k,j}}{1-\lambda_k} <v_k, v_1>  \nonumber \\
&= \frac{2 |\widetilde{E}|}{d_j}  \sum_{k=2}^n \frac{1}{1-\lambda_k} v_{k,j}^2  \nonumber \\
&=\frac{1}{\pi(j)} \sum_{k=2}^n \frac{1}{1-\lambda_k} v_{k,j}^2 \nonumber
\end{align}
Similarly we obtain,
\begin{align}
H^i&=\sum_{j=1}^n  \pi(j) H_{ij} \nonumber =
\sum_{j=1}^n \sum_{k=2}^n \frac{1}{1-\lambda_k} \left(  v_{k, j}^2 \frac{d_i}{d_j} - v_{k,i}  v_{k, j}
\sqrt{\frac{d_j}{d_i}} \right)  \nonumber \\
&=
\sum_{k=2}^n \frac{1}{1-\lambda_k} \left(  \sum_{j=1}^n v_{k, j}^2  - v_{k,i}
\sqrt{\frac{1}{d_i}}  \sum_{j=1}^n v_{k, j} \sqrt{d_j} \right)
=   \sum_{k=2}^n \frac{1}{1-\lambda_k}  \nonumber
\end{align}
Note, that by orthogonality  
we have $$\sum_{k=2}^n v_{k,j}^2 = 1-v_{1,j}^2=1- \pi(j).$$
On the other hand $$  \sum_{k=2}^n (1-\lambda_k) v_{k,j}^2 = \sum_{k=1}^n (1-v_k) v_{k, j}^2= 1- B_{jj} = 1$$ since
$B= \sum_{k=1}^n \lambda_k v_k v_k^t$ (by the spectral theorem and the fact that the  adjacency matrix has zeros on the diagonal). Therefore, employing the inequality between arithmetic and harmonic means
$$
\frac{\sum_{k=2}^n \frac{1}{1-\lambda_k} v_{k,j}^2  }{\sum_{k=2}^n  v_{k,j}^2}   \geq \frac{\sum_{k=2}^n  v_{k,j}^2}{ \sum_{k=2}^n (1-\lambda_k)v_{k,j}^2 }.
$$
Thus
\begin{align}
H_j&= \frac{1}{\pi(j)} \sum_{k=2}^n \frac{1}{1-\lambda_k} v_{k, j}^2
\geq \frac{1}{\pi(j)} \frac{(\sum_{k=2}^n v_{k, j}^2)^2}{\sum_{k=2}^n (1-\lambda_k) v_{k,j}^2} =\frac{1}{\pi(j)} (1-\pi(j))^2
\ge \frac{2|\widetilde{E}|}{d_j} - 2 \nonumber
\end{align}
On the other hand,
\begin{align}
H_j&= \frac{2|\widetilde{E}|}{d_j} \sum_{k=2}^n \frac{1}{1-\lambda_k} v_{k,j}^2 \leq \frac{2|\widetilde{E}|}{d_j}  \frac{1}{1-\lambda_2} (1-\pi(j))=  \frac{2|\widetilde{E}|}{d_j}  \frac{1}{1-\lambda_2} (1-\frac{d_j}{2|\widetilde{E}|}). \nonumber
 \end{align}
It thus suffices to analyze the behaviour of $|\widetilde{E}|$, $d_j$, and the size of the spectral gap $1-\lambda_2$.\\
For the first two quantitites, consider any vertex $j \in H$. Then  $$d_j= (d-1) \sharp \{ e: j \in e\} \quad \mbox{i.e.}\quad
d_j= \sum\limits_{\substack{i_1<i_2<\cdots <i_{d-1} \\ i_k\neq j; \,\, \forall \,\, k=1, \cdots, d-1}} X_{i_1, i_2, \cdots, i_{d-1}, j}$$
where $X_{i_1, \cdots, i_d}$ is the indicator for the presence of the edge $(i_1, \cdots, i_d)$. Note that $\mathbb{E}(d_j) = \binom{n}{d-1}  p$ tends to $\infty $ by definition of $p$.
By Chernoff's inequality:
$$ \mathbb{P}(d_j \le \mathbb{E}(d_j)-\lambda) \leq e^{-\frac{\lambda^2}{2\mathbb{E}(d_j)}}
\quad \mbox{and} \,\,\, \mathbb{P}(d_j \ge \mathbb{E}(d_j)+\lambda) \leq e^{-\frac{\lambda^2}{2\mathbb{E}(d_j) + \frac{\lambda}{3}}}$$
Choosing $\lambda = c \sqrt{\binom{n}{d-1} p }$ for some constant $c>0$ leads to:
 \begin{align}
 \mathbb{P}\left( \mathbb{E}(d_j)-\lambda < d_j < \mathbb{E}(d_j)+\lambda\right) &=
 \mathbb{P}\left( \{\{d_j \le \mathbb{E}(d_j)-\lambda\} \cup \{d_j \ge \mathbb{E}(d_j)+\lambda\}\}^c\right)  \nonumber \\
 &=1- \mathbb{P}\left( \{d_j \le \mathbb{E}(d_j)-\lambda\} \cup \{d_j \ge \mathbb{E}(d_j)+\lambda\}\}\right)  \nonumber \\
 &\ge 1-  \mathbb{P}\left( d_j \le \mathbb{E}(d_j)-\lambda\right)- \mathbb{P}\left( d_j \ge \mathbb{E}(d_j)+\lambda\right)\nonumber \\
  &\ge 1-  e^{-\frac{\lambda^2}{2\mathbb{E}(d_j)}}- e^{-\frac{\lambda^2}{2\mathbb{E}(d_j) + \frac{\lambda}{3}}}
 \nonumber \\
 &\ge 1-  e^{-\frac{c^2}{2}}- e^{-\frac{c^2}{4}} \ge  1- 2e^{-\frac{c^2}{4}}  \nonumber
 \end{align}
 for $n$ sufficiently large.

On the other hand,$|\widetilde{E}| = \binom{d}{2} \sharp \{e: e \in E  \}$ where $E$ is the set of hyperedges.
Thus $\mathbb{E}(|\widetilde{E}|) = \binom{d}{2} \binom{n}{d} p $.
If we consider a deviation of $c \sqrt{\binom{d}{2} \binom{n}{d} p}$ for some $c>0$ we again obtain by an application of Chernoff's inequality as above that with probability $1-2e^{\frac{c^2}{4}}$
$$
\binom{d}{2} \binom{n}{d} p-c \sqrt{\binom{d}{2} \binom{n}{d} p}<\widetilde{E}<\binom{d}{2} \binom{n}{d} p+c \sqrt{\binom{d}{2} \binom{n}{d} p}
$$
If we choose $c=\log n$ we obtain that for every fixed $j$ with probability at least $1-4e^{-\frac{(\log n)^2}{4}}$:
\begin{align}
\frac{2|\widetilde{E}|}{d_j} &\leq \frac{ 2 \binom{d}{2}  \binom{n}{d} +\log n \sqrt{ \binom{d}{2} \binom{n}{d} p } }{ (d-1) \binom{n}{d-1}p - \log n \sqrt{n^{d-1} p}    } = n(1 + o(1)) \nonumber
\end{align}
(due to our choice of $p$). Similarly we see that $\frac{2|\widetilde{E}|}{d_j} \geq n(1-o(1))$ with probability at least $1-4e^{-\frac{(\log n)^2}{4}}$. Since $n e^{-\frac{(\log n)^2}{4}}$ converges to $0$, we see that
$\frac{2|\widetilde{E}|}{d_j} = n(1+o(1))$ a.a.s. simultaneously for all $j$.
%

Now, we turn to the spectral gap. Fortunately most of the work has already has been done by Lu and Peng (see
\cite{LuPeng12}) consider $d$-uniform hypergraphs $H$ and for every pair of sets $I$ and $J$ with
cardinality $s$ they associate a weight $w(I, J)$, which is the number of edges in $H$ passing through
$I$ and $J$ if $I \cap J = \emptyset$, and $0$, otherwise. The $s$-th Laplacian of $H$
is defined to be the normalized Laplacian of the thus obtained weighted graph.
As a special case, for $s=1$ we can thus consider the Laplacian  $L_A := I - D^{\frac{1}{2}} A D^{\frac{1}{2}}.$
As shown in \cite{LuPeng12} the ordered eigenvalues of $L_A$ fulfill
$$0= \tilde{\lambda}_0 \leq \tilde{\lambda}_1 \leq \cdots \leq \tilde{\lambda}_{n-1} \leq 2 $$
and:
\begin{theo}(cf. \cite[Theorem 2]{LuPeng12} of which this is a special case) \label{LuPeng}
Denote by $\overline{\lambda} = \max \{ 1-  \tilde{\lambda}_1, \tilde{\lambda}_{n-1}-1  \} = \overline{\lambda}(H^d(n, p)).$
If $p(1-p) \gg  \frac{\log^4 n}{n}$ and $1-p \gg  \frac{\log n}{n^2}$ then a.a.s.
$$ \overline{\lambda}(H^d(n, p))  \leq  \frac{1}{n-1}+ (3+o(1)) \sqrt{ \frac{1-p}{ \binom{n-1}{d-1} p } }.
$$
\end{theo}

\begin{rem}
The second condition on $p$, $1-p \gg  \frac{\log n}{n^2}$, may be omitted for our purposes
because just serves to control the
smallest eigenvalue of $L_A$.
Also note that  $\sqrt{  \frac{1-p}{ \binom{n-1}{d-1} p }}$
is at most of order $ \frac 1{\log^2 n}$.
\end{rem}

Translated to our problem, Theorem \ref{LuPeng} implies that the eigenvalues $\lambda_1, \cdots, \lambda_{n-1}$ for the matrix
$D^{\frac{1}{2}} A D^{\frac{1}{2}} = I-L_A$
satisfy $\lambda_1=1$ and $$1-\lambda_2 \geq 1- \frac{1}{n-1}-(3+o(1))\sqrt{  \frac{1-p}{ \binom{n-1}{d-1} p }}.$$
Thus we get the following upper bound
\begin{cor}
If $p(1-p) \gg  \frac{\log^4 n}{n^{d-1}}$ a.a.s.
$$\frac{1}{1-\lambda_2} \leq \frac{1}{1-\frac{1}{n-1}-(3+o(1))\sqrt{  \frac{1-p}{ \binom{n-1}{d-1} p }}} = 1+ o(1)$$
\end{cor}
Thus we have seen
\begin{theo}
If $p(1-p) \gg  \frac{\log^4 n}{n^{d-1}}$
then a.a.s.
$$H_j=n(1+o(1)).$$
%
\end{theo}
On the other hand, we have already seen that $H^i = \sum_{k=2}^n \frac{1}{1-\lambda_k}$.
We therefore obtain
\begin{theo}
If $p(1-p) \gg  \frac{\log^4 n}{n^{d-1}}$ then a.a.s. $$ \frac 12 n(1+o(1))\le H^i \le n(1+o(1)).$$
\end{theo}

\begin{proof}
The key observation is that under the given conditions we have that
$$ \frac{1}{2} \leq  \frac{1}{1-\lambda_k} \leq  \frac{1}{1-\lambda_2}=1+o(1)$$ for all $k$. This proves the assertion.
\end{proof}

\section{Commute times and Cover times}
We turn now to the study of the commute time $\kappa(i, j) = H_{ij} + H_{ji}.$
An elementary computation using Theorem \ref{Lovasztheo} gives that
$$
\kappa(i, j) = 2 |\tilde{E}|  \sum_{k=2}^n \frac{1}{1-\lambda_k} \left(\frac{v_{k,i}}{\sqrt{d_i}} - \frac{v_{k, j}}{\sqrt{d_j}}\right)^2
$$
(also see \cite[Corollary 3.2]{{lovaszgraphs}}). Using this representation we obtain:
\begin{prop}
For all $i,j \in V$ we obtain the following bounds for the commute time
$$ |\tilde{E}| \left( \frac{1}{d_i}+ \frac{1}{d_j} \right)  \leq\kappa(i, j) \leq \frac{2 |\tilde{E}|}{1-\lambda_2}  \left( \frac{1}{d_i} + \frac{1}{d_j} \right).$$
\end{prop}

\begin{proof}
The proof follows the ideas in the of an unweighted simple graph (see \cite{lovaszgraphs}). Again
$ \frac{1}{2} \leq  \frac{1}{1-\lambda_k} \leq  \frac{1}{1-\lambda_2}.$ Hence
$$
|\tilde{E}|  \sum_{k=2}^n  \left(\frac{v_{k,i}}{\sqrt{d_i}} - \frac{v_{k, j}}{\sqrt{d_j}}\right)^2 \le  \kappa(i, j)  \le
2 |\tilde{E}|  \frac{1}{1-\lambda_2} \sum_{k=2}^n \left(\frac{v_{k,i}}{\sqrt{d_i}} - \frac{v_{k, j}}{\sqrt{d_j}}\right)^2.
$$
But
\begin{align}
    \sum_{k=2}^n \left(  \frac{v_{k, i}}{\sqrt{d_i}} -  \frac{v_{k, j}}{\sqrt{d_j}} \right)^2 &=  
   \frac{1-\pi(i)}{d_i} + \frac{1-\pi(j)}{d_j} -  2 \sum_{k=1}^n  \frac{v_{k, i} v_{k, j}}{\sqrt{d_i d_j}}  + 2  \frac{v_{1, i} v_{1, j}}{\sqrt{d_i d_j}}   \nonumber  \\
    &= \frac{1}{d_i} + \frac{1}{d_j} - \frac{1}{2|\tilde{E}|}- \frac{1}{2|\tilde{E}|}
     +2  \frac{ \sqrt{ \frac{d_i}{2|\tilde{E}|} }\sqrt{\frac{d_j}{2|\tilde{E}|}} }{ \sqrt{d_i  d_j}} \nonumber \\&=
    \frac{1}{d_i} + \frac{1}{d_j}. \nonumber
    \end{align}

  \end{proof}
This gives the following bound on $\kappa(i, j)$.
\begin{theo}\label{theocover}
For $p(1-p) \gg  \frac{\log^4 n}{n^{d-1}}$ a.a.s. in $i$ and $j$
$$n(1+o(1)) \leq \kappa(i, j) \leq 2 n (1+o(1)).$$
\end{theo}
Finally, we also want to give a bound the cover time $C(H)$. From Theorem $2.7$ in Lov\'az (see \cite{lovaszgraphs}) we have that:
\begin{theo}
The cover time from any vertex $i$  of a graph with $n$ vertices is bounded as follows:
$$\min_{i, j} H_{i, j} \sum_{k=1}^n \frac{1}{k} \leq C(H) \leq \max_{i, j} H_{i, j} \sum_{k=1}^n \frac{1}{k}$$
\end{theo}
Thus we obtain
\begin{theo}
For $p(1-p) \gg  \frac{\log^4 n}{n}$
we have a.a.s \quad
$\frac{n}{2}\log n \leq C(H) \leq n \log n.$
\end{theo}

\begin{proof}
By \eqref{specrep} and $ \frac{1}{2} \leq \frac{1}{1-\lambda_k} \leq \frac{1}{1-\lambda_2} $ we get:
$$ |\tilde{E}|\sum_{k=2}^{n} \left( \frac{v^2_{k, j}}{d_j} - \frac{v_{k, i} v_{k, j}}{ \sqrt{d_i d_j}} \right) \leq   H(i, j)  \leq \frac{ 2|\tilde{E}|}{1-\lambda_2}  \sum_{k=2}^{n} \left( \frac{v^2_{k, j}}{d_j} - \frac{v_{k, i} v_{k, j}}{ \sqrt{d_i d_j}} \right)$$
On the other hand:
\begin{align}
\sum_{k=2}^{n} \left( \frac{v^2_{k, j}}{d_j} - \frac{v_{k, i} v_{k, j}}{ \sqrt{d_i d_j}} \right) &=
\sum_{k=1}^{n} \left( \frac{v^2_{k, j}}{d_j} - \frac{v_{k, i} v_{k, j}}{ \sqrt{d_i d_j}} \right) -  \frac{v^2_{1, j}}{d_j} + \frac{v_{1, i} v_{1, j}}{ \sqrt{d_i d_j}} \nonumber \\
&=  \frac{1}{d_j} - \frac{ \frac{d_j}{2 |\tilde{E}| } }{d_j} +\frac{ \sqrt{ \frac{d_i}{2 |\tilde{E}|}   }   \sqrt{\frac{d_j}{2 |\tilde{E}|}}  }{ \sqrt{d_i d_j} } = \frac{1}{d_j}. \nonumber
\end{align}
Now $\frac{1}{1-\lambda_2}=1+o(1)$ and $\frac{2|\tilde E|}{d_j}=n(1+o(1))$ a.a.s. {\it uniformly} in $j$. This, together with $\sum_{k=1}^n \frac{1}{k} \sim \log n$ finishes the proof.
\end{proof}
\begin{rem}
We remark that the vertex cover time $C(H)$ in the case of random walk on  $d$-uniform hypergraphs is smaller than the vertex cover time in the case of random walk on $r$-regular $d$-uniform hypergraphs.
\end{rem}



%
%
\end{document}